\documentclass[11pt]{article} 

\usepackage[utf8]{inputenc} 

\usepackage{geometry} 
\usepackage{graphicx} 

\usepackage{amsmath}

\usepackage{amssymb}

\newtheorem{theorem}{Theorem}[section]
\newtheorem{lemma}[theorem]{Lemma}
\newtheorem{proposition}[theorem]{Proposition}

\newenvironment{proof}[1][Proof]{\begin{trivlist}
\item[\hskip \labelsep {\bfseries #1}]}{\end{trivlist}}
\newenvironment{definition}[1][Definition]{\begin{trivlist}
\item[\hskip \labelsep {\bfseries #1}]}{\end{trivlist}}

\newcommand{\qed}{\nobreak \ifvmode \relax \else
      \ifdim\lastskip<1.5em \hskip-\lastskip
      \hskip1.5em plus0em minus0.5em \fi \nobreak
      \vrule height0.75em width0.5em depth0.25em\fi}

\title{The Lucas Numbers and Other Gibonacci Sequences Mod m}
\author{Jeremiah T. Southwick \\ Le Moyne College \\ southwjt@lemoyne.edu}

\begin{document}
\maketitle

\pagebreak

\begin{abstract}
We study Gibonacci sequences mod $m$, giving special attention to the Lucas numbers. It is known which $m$ have the property that the Fibonacci sequence contains all residues mod $m$. When $m$ has this property, we say that the Fibonacci sequence is complete mod $m$. We extend this work to all Gibonacci sequences, concluding by determining the set of $m$ in which a generic Gibonacci sequence containing the relatively prime consecutive terms $a, b$ is complete mod $m$.
\end{abstract}

\pagebreak

\tableofcontents

\pagebreak

\section{Introduction \& Definitions}

This paper is the result of a senior research project conducted in the summer and fall of 2013. All work was done independently of the similar result in ~\cite{Lang2013}.

Burr~\cite{Burr71} has completely categorized which moduli $m$ have the property that the Fibonacci sequence contains all residues mod $m$. If $m$ has this property, he says that the Fibonacci sequence is complete mod $m$, and if $m$ does not have this property, he says the Fibonacci sequence is defective mod $m$. The set of $m$ in which the Fibonacci sequence is complete is $\{5^k, 2 \cdot 5^k, 4 \cdot 5^k, 3^j \cdot 5^k, 6 \cdot 5^k, 7 \cdot 5^k, 14 \cdot 5^k\}$ where $k \geq 0$ and $j \geq 1$. The goal of this paper is to extend the idea of completeness to any numerical sequence satisfying the Fibonacci relation, to obtain a categorization of which $m$ have the property that a generic sequence contains all residues mod $m$. We begin by defining these sequences, which we call Gibonacci sequences, for generalized Fibonacci.

\begin{definition}
A \emph{Gibonacci sequence} is the numerical sequence $\{G_n(a, b)\}$ where $a, b \in \mathbb{Z}$, satisfying the three following conditions:
\begin{enumerate}
  \item $\gcd(a, b) = 1$.
  \item $G_1(a, b) = a$, $G_2(a, b) = b$.
  \item $G_{n+1}(a, b) = G_{n-1}(a, b) + G_n(a, b)$.
\end{enumerate}
\end{definition}

The Fibonacci sequence (\{1, 1, 2, 3, 5, 8, ...\}) is the Gibonacci sequence $\{G_n(1, 1)\}$, and the Lucas numbers (\{1, 3, 4, 7, 11, 18, ...\}) are the Gibonacci sequence $\{G_n(1, 3)\}$. We will reserve $F_n$ and $L_n$ to denote these sequences, respectively.

We require the condition that $a$ and $b$ be relatively prime so that the set of Gibonacci sequences does not contain any sequences which are multiples of each other. Without this requirement we would include sequences like \{2, 2, 4, 6, 10, ...\}, which is just the Fibonacci sequence multiplied by 2. But that sequence would automatically be defective mod 2, since no odd number appears. By requiring $\gcd(a,b) = 1$, we remove that possibility, making for more concise statements throughout the paper.

As with many other mathematical concepts, when a Gibonacci sequence is considered mod $m$, it becomes cyclical. Since this proposition is assumed as true in the existing literature we have studied, we present our own proof of the fact here.

\begin{proposition}
\label{GibCyclical}
Gibonacci sequences are cyclical mod $m$.
\end{proposition}

\begin{proof}
Since there are a finite number of residues mod $m$, there are a finite number of ways to place them in pairs next to each other. Thus in the infinite string of residues representing a Gibonacci sequence mod $m$, some pair $G_k$ (mod $m$), $G_{k+1}$ (mod $m$) must appear more than once. Therefore we know a Gibonacci sequence mod $m$ takes the form \{..., $G_k$ (mod $m$), $G_{k+1}$ (mod $m$), ..., $G_r$ (mod $m$), $G_{r+1}$ (mod $m$), ...\} where $G_k$ (mod $m$) $\equiv$ $G_r$ (mod $m$) and $G_{k+1}$ (mod $m$) $\equiv$ $G_{r+1}$ (mod $m$). By definition, $G_{n+1} = G_{n-1} + G_n$, so beginning  with $G_k$ (mod $m$) and $G_{k+1}$ (mod $m$) we can follow this rule until our adding brings us to $G_r$ (mod $m$), $G_{r+1}$ (mod $m$). But these are the first two numbers we began our addition with, so any additional numbers will just be repetitions of the string we have already generated. We can also go backwards from $G_k$ (mod $m$) and $G_{k+1}$ (mod $m$) with the reverse Fibonacci relation, $G_{n-1} = G_{n+1} - G_n$, deriving the same string in reverse order until we reach $G_2$ (mod $m$) and $G_1$ (mod $m$), since $G_{k-h} \equiv G_{r-h}$ (mod $m$) for $0 \leq h \leq k$. In this way, we can see that the entirety of $\{G_n(a,b)\}$ mod $m$ is the same string repeated. Thus Gibonacci sequences are cyclical mod $m$. $\blacksquare$ 
\end{proof}

Because of this property, it is convenient to define Gibonacci cycles mod $m$.

\begin{definition}
We say the finite sequence of integers $(i_1, i_2, ..., i_{h-1}, i_h)$ (mod $m$) is a \emph{Gibonacci cycle mod m} if it satisfies the following conditions:

\begin{enumerate}
  \item $i_{n+2} \equiv i_n + i_{n+1}$ (mod $m$) for $n = 1, 2, ..., h-2$.
  \item $i_{h-1} + i_h \equiv i_1$ (mod $m$) and $i_h + i_1 \equiv i_2$ (mod $m$).
  \item No $n < h$ satisfies $i_{n-1} + i_n \equiv i_1$ (mod $m$) and $i_n + i_1 \equiv i_2$ (mod $m$).
\end{enumerate}
\end{definition}

Condition 1 ensures that the main body of the cycle satisfies the Fibonacci relation, condition 2 ensures that the same relation is maintained at the ends of the cycle, and condition 3 ensures that the cycle doesn't contain any sub-cycles within it.

For example, (1, 1, 2, 3, 1, 0) is a Gibonacci cycle mod 4. We can see that it corresponds to the Fibonacci sequence, because the Fibonacci sequence mod 4 is \{1, 1, 2, 3, 1, 0, 1, 1, ...\}.

We note that (0) is a Gibonacci cycle mod $m$ for all $m$. We refer to it as the \emph{trivial Gibonacci cycle}.

It is also important to note that not all Gibonacci cycles are necessarily derived from some Gibonacci sequence. For example, (2, 2, 4, 0, 4, 4, 2, 0) is a Gibonacci cycle mod 6, but because all $i_n$ within the cycle are divisible by 2, any sequence deriving this cycle will have only even numbers, meaning it will not satisfy our requirement that $\gcd(a,b) = 1$, and hence will not be a Gibonacci sequence.

We next define what it means for two Gibonacci cycles to be equivalent.

\begin{definition}
We say two Gibonacci cycles $(i_1, i_2, ..., i_{h-1}, i_h)$ (mod $m$) and $(j_1, j_2, ..., j_{h'-1}, j_{h'})$ (mod $m$) are \emph{equivalent} if

\begin{enumerate}
\item $h = h'$.
\item For some $0 \leq n, r \leq h$, $i_n = j_r$ and $i_{n+1 (mod \ h)} = j_{r+1 (mod \ h)}$.
\end{enumerate}
\end{definition}

Thus (1, 1, 2, 3, 1, 0) and (2, 3, 1, 0, 1, 1) are equivalent Gibonacci cycles mod 4, because they both have the same lengths and contain the same numbers, just shifted over in the cycle. We make this definition so that for each $m$, we can create the set of all inequivalent Gibonacci cycles mod $m$.

\begin{definition}
A \emph{complete Gibonacci system mod $m$} is the set of all inequivalent Gibonacci cycles mod $m$.
\end{definition}

For example, \{(1, 1, 0), (0)\} is the complete Gibonacci system mod 2. Complete Gibonacci systems are useful in our endeavors, as they allow us to consider all Gibonacci cycles mod $m$ simultaneously and hence all Gibonacci sequences mod $m$.

We depart here in our terminology from Burr, who used the phrases ``Fibonacci cycle mod $m$'' and ``complete Fibonacci system mod $m$'', but we mean the same thing as he did. We do this to reflect our differing usages of the words ``Fibonacci'' and ``Gibonacci''. We will use our phrases for the duration of the paper.

It was noted by Burr that the total number of terms (i.e., individual digits) appearing in a complete Gibonacci system mod $m$ is $m^2$. We present our own proof of this proposition here.

\begin{proposition}
\label{MSquared}
The total number of terms appearing in a complete Gibonacci system mod $m$ is $m^2$.
\end{proposition}

\begin{proof}
If we form the set of all pairs $[a, b]$ where $0 \leq a, b \leq m-1$, there will be a total of $m^2$ pairs, representing the possible ways to place all least positive residues mod $m$ next to each other. These least positive residues are the terms in question, totaling $2m^2$ since there are 2 terms in each pair. We can combine these pairs to form cycles by merging $[a, b]$ and $[b, c]$ if $a + b \equiv c$ (mod $m$). In the case of the pair $[0, 0]$, we merge it with itself to form the trivial Gibonacci cycle (0). In all other cases, each pair will uniquely merge with one other pair since $a + b$ (mod $m$) equals a unique $c$ (mod $m$). Each time a merge is made, the number of terms in our set will be reduced by one. Because there are $m^2$ pairs, after all our merging is done we will have $m^2$ fewer terms than we began with. Therefore there are $2m^2 - m^2 = m^2$ total terms in a complete Gibonacci system mod $m$. $\blacksquare$
\end{proof}

Thus if we find a collection of inequivalent Gibonacci cycles mod $m$ totaling $m^2$ terms, we know it is a complete Gibonacci system mod $m$.

\section{Results for all Gibonacci sequences}

The most famous Gibonacci sequence is the Fibonacci sequence. Burr shows that this sequence is complete mod $m$ for all numbers of the form \{$5^k, 2 \cdot 5^k, 4 \cdot 5^k, 3^j \cdot 5^k, 6 \cdot 5^k, 7 \cdot 5^k, 14 \cdot 5^k$\} where $k \geq$ 0 and $j \geq$ 1, and is defective for all other numbers. We will extend these results to all Gibonacci sequences both by adapting Burr's results for the Fibonacci sequence to generic sequences and by establishing our own results. We begin by introducing the terminology of a \emph{multiple of the Fibonacci sequence mod $m$}.

\begin{definition}
We say that a Gibonacci sequence \{$G_n(a,b)$\} mod $m$ is a \emph{multiple of the Fibonacci sequence mod $m$} if there is some $k \in \mathbb{N}$ by which the Fibonacci cycle \\ $(F_1, F_2, …, F_{h-1}, F_h)$ (mod $m$) can be multiplied, so that $(k \cdot F_1, k \cdot F_2, …, k \cdot F_{h-1}, k \cdot F_h)$ (mod $m$) is equivalent to $(G_1, G_2, . . ., G_{h-1}, G_h)$ (mod $m$).
\end{definition}

We next show a useful application of this definition:

\begin{proposition}
\label{ZeroEquivMultiple}
There is a 0 in \{$G_n(a,b)$\} mod $m$ if and only if \{$G_n(a,b)$\} mod $m$ is a multiple of the Fibonacci sequence mod $m$.
\end{proposition}

\begin{proof}
Suppose that 0 appears in the cycle $(G_1, G_2, …, G_{h-1}, G_h)$ (mod $m$). Then there is a number directly after it in the cycle (possibly wrapped around the ends), which we will call $k'$. After this number is $k'$, since 0 + $k' = k'$. If we multiply the Fibonacci cycle mod $m$ by this number $k'$, we will derive \{$G_n(a,b)$\} mod $m$, since the Fibonacci cycle mod $m$ contains the pair [1, 1] and because if $a + b \equiv c$ (mod $m$) then $k'a + k'b \equiv k'c$ (mod $m$). Thus \{$G_n(a,b)$\} mod $m$ is a multiple of the Fibonacci sequence mod $m$.

Conversely, suppose that \{$G_n(a,b)$\} mod $m$ is a multiple of the Fibonacci sequence mod $m$. Because the Fibonacci cycle mod $m$ contains a 0 (since [1, 1] is in the cycle and 0 + 1 = 1), we know that \{$G_n(a,b)$\} mod $m$ will contain a 0, since $0 \cdot k \equiv 0$ (mod $m$) for all $k$ and $m$. $\blacksquare$
\end{proof}

This proposition is significant, because it relates a generic Gibonacci sequence to the Fibonacci sequence. Hence we are able to use it to get results for all Gibonacci sequences. In fact, our next proposition will do just that.

\begin{proposition}
\label{FibToGib}
Suppose the Fibonacci sequence is defective mod $m$ for some $m$. Then all Gibonacci sequences are defective mod $m$.
\end{proposition}

\begin{proof}
We will assume the Fibonacci sequence is defective mod $m$ and proceed by cases:

Case 1: If 0 does not appear in \{$G_n(a,b)$\} mod $m$, then \{$G_n(a,b)$\} is defective mod $m$.

Case 2: If 0 appears in \{$G_n(a,b)$\} mod $m$, then \{$G_n(a,b)$\} mod $m$ is a multiple of the Fibonacci sequence mod $m$ and there is a $k$ such that $G_n \equiv k \cdot F_{n+r}$ (mod $m$) for all $n$ and some $r$. Thus the number of residues that appear in \{$G_n(a,b)$\} mod $m$ must be less than or equal to the number of residues that appear in the Fibonacci sequence mod $m$. Because the Fibonacci sequence is defective mod $m$, $\{G_n(a,b)\}$ must also be defective mod $m$. $\blacksquare$
\end{proof}

We will also make use of the following proposition, a result of Shah~\cite{Shah68}. We present his proof here.

\begin{proposition}
\label{mTotm}
If \{$G_n(a,b)$\} is defective mod $m$, then it will also be defective mod $tm$, where $t \in \mathbb{N}$.
\end{proposition}

\begin{proof}
The proof of the contrapositive is as follows: Suppose \{$G_n(a,b)$\} is complete mod $tm$. Then for all $0 \leq r \leq m-1$, some $G_n \equiv r$ (mod $tm$). But then $G_n \equiv r$ (mod $m$), meaning \{$G_n(a,b)$\} is also complete mod $m$. $\blacksquare$
\end{proof}

\section{The Lucas Numbers Mod $m$}

In this section we characterize which $m$ have the property that the Lucas numbers are complete mod $m$.

\begin{theorem}
\label{CompleteMods}
The Lucas numbers are complete mod $m$ if and only if $m$ is one of the following numbers: 2, 4, 6, 7, 14, or 3$^j$, where $j \geq$ 1.
\end{theorem}

Theorem~\ref{CompleteMods} will be proved through a series of lemmas. In comparison to the Fibonacci sequence, which is complete mod $5^k, 2 \cdot 5^k, 4 \cdot 5^k, 3^j \cdot 5^k, 6 \cdot 5^k, 7 \cdot 5^k,$ and $14 \cdot 5^k$ for $k \geq$ 0 and $j \geq$ 1, this is a relatively small number of moduli.

If we remove the 5$^k$s from the list of moduli $m$ in which the Fibonacci sequence is complete mod $m$, we derive the numbers in theorem~\ref{CompleteMods}. This observation leads us to believe that 5 will be important to our discussion, so we will consider the Lucas numbers mod 5.

\begin{lemma}
\label{FiveDefective}
The Lucas numbers are defective mod 5.
\end{lemma}

\begin{proof}
This is seen by direct analysis of the Lucas numbers mod 5. The Lucas cycle mod 5 is (1, 3, 4, 2). The residue 0 never appears, meaning the Lucas numbers are defective mod 5. $\blacksquare$
\end{proof}

	Furthermore, by proposition~\ref{mTotm}, the defectiveness of the Lucas numbers mod 5 implies that the Lucas numbers will also be defective mod $5t$, for all multiples of 5. Thus we can remove any factors of 5 from the list of mods $m$ in which the Fibonacci sequence is complete and limit our searching for mods $m$ in which the Lucas numbers are complete to whatever remains. These remaining numbers are exactly the numbers listed in theorem~\ref{CompleteMods}. We will show by direct observation that the Lucas numbers are complete mod 2, 4, 6, 7, and 14.

\begin{lemma}
\label{FourComp}
The Lucas numbers are complete mod 4, 6, and 14.
\end{lemma}

\begin{proof}
The Lucas cycle mod 4 is (1, 3, 0, 3, 3, 2). Thus all residues mod 4 appear in the Lucas numbers. The Lucas cycle mod 6 is (1, 3, 4, 1, 5, 0, 5, 5, 4, 3, 1, 4, 5, 3, 2). Thus all residues mod 6 appear in the Lucas numbers. The Lucas cycle mod 14 is

\[ \begin{array}{cccccccc}
\mbox{(1,} & 3, & 4, & 7, & 11, & 4, & 1, & 5, \\
\mbox{6,} & 11, & 3, & 0, & 3, & 3, & 6, & 9, \\
\mbox{1,} & 10, & 11, & 7, & 4, & 11, & 1, & 12, \\
\mbox{13,} & 11, & 10, & 7, & 3, & 10, & 13, & 9, \\
\mbox{8,} & 3, & 11, & 0, & 11, & 11, & 8, & 5, \\
\mbox{13,} & 4, & 3, & 7, & 10, & 3, & 13, & 2). \end{array}\]

Thus all residues mod 14 appear in the Lucas numbers, and the lemma is shown. $\blacksquare$
\end{proof}

Because the Lucas numbers are complete mod 14, we also know by proposition~\ref{mTotm} that the Lucas numbers are complete mod 2 and 7, as well.

Hence it only remains to be shown that the Lucas numbers are complete mod $3^j$ for all $j \geq 1$, and theorem~\ref{CompleteMods} will be proven. We will prove a much stronger lemma which claims that any Gibonacci sequence \{$G_n(a,b)$\} is complete mod $3^j$ for all $j \geq 1$.

Burr uses induction to construct complete Gibonacci systems mod $3^j$. The $j = 1$ case gives the system $\{(0, 1, 1, 2, 0, 2, 2, 1), (0)\}$. Then assuming a complete Gibonacci system has been determined mod $3^{j-1}$ for $j \geq 2$, he constructs a complete Gibonacci system mod $3^j$ as follows: He determines the Fibonacci cycle mod $3^j$ and then creates the complete Gibonacci system mod $3^j$ in two parts. The first part is $k$ times the Fibonacci cycle mod $3^j$, where $k < 3^j$ and $\gcd(3,k) = 1$. There are $3^{j-1}$ of these, and each is denoted as $C_k$. The second part is formed by multiplying each term in the inductively hypothesized complete Gibonacci system mod $3^{j-1}$ by 3. These two parts are disjoint, containing a total of $3^{2j}$ terms and hence together form a complete Gibonacci system mod $3^j$. We will use this information to prove the following lemma:

\begin{lemma}
\label{ThreeNComp}
All Gibonacci sequences \{$G_n(a,b)$\} are complete mod $3^j$ for all $j \geq 1$.
\end{lemma}

\begin{proof}
Burr showed that the Fibonacci sequence is complete mod $3^j$. We are not concerned here with the intricacies of Burr’s argument, but we wish to determine where \{$G_n(a,b)$\} mod $3^j$ appears in the complete Gibonacci system mod $3^j$. This cycle must appear somewhere within the system, since it is a complete system. \{$G_n(a,b)$\} contains the relatively prime $a$ and $b$ as consecutive terms by definition. Hence, $a$ (mod $3^j$) and $b$ (mod $3^j$) are not both multiples of 3, for otherwise $a$ and $b$ would not be relatively prime, and so $\{G_n(a,b)\}$ mod $3^j$ is not entirely made up of multiples of 3. Hence, it cannot be a cycle that was constructed by multiplying the complete Gibonacci system mod $3^{j-1}$ by 3, as all of those terms are multiples of 3. This means that \{$G_n(a,b)$\} mod $3^j$ must appear in some $C_k$, where $k$ is relatively prime to 3 and therefore relatively prime to $3^j$. Hence, the Gibonacci cycle $(k \cdot F_1, k \cdot F_2, …, k \cdot F_{h-1}, k \cdot F_h)$ (mod $3^j$) contains the same number of least positive residues as the Fibonacci cycle mod $3^j$, because we are guaranteed unique inverses by $\gcd(k, 3^j) = 1$. Since the Fibonacci sequence is complete mod $3^j$, any Gibonacci sequence \{$G_n(a,b)$\} must also be complete mod $3^j$. $\blacksquare$
\end{proof}

\begin{proof}[Proof of Theorem~\ref{CompleteMods}]
Having shown lemma~\ref{ThreeNComp}, we have completed our proof of theorem~\ref{CompleteMods}. The Lucas numbers are complete mod 2, 4, 6, 7, 14, and $3^j$, where $j \geq 1$ and they are defective mod $m$ for all $m$ not of this form. $\blacksquare$
\end{proof}

\section{Extending completeness in the Lucas Numbers}

Lemma~\ref{ThreeNComp} makes us wonder if we can extend the completeness of the Lucas numbers mod $m$ to all Gibonacci sequences no matter what $m$ is. If this is the case, then the set of $m$ in which the Lucas numbers are complete mod $m$ will be a subset of the set of $m$ in which any generic Gibonacci sequence is complete mod $m$. We offer this as a proposition below:

\begin{proposition}
\label{LucasMin}
If the Lucas numbers are complete mod $m$, then all Gibonacci sequences are complete mod $m$.
\end{proposition}

We will construct several complete Gibonacci systems to prove proposition~\ref{LucasMin}. We need only check the complete Gibonacci systems mod 4, 6, and 14 to succeed in our proof, as 2 and 7 will follow from 14.

\begin{lemma}
\label{FourGibComp}
All Gibonacci sequences are complete mod 4.
\end{lemma}

\begin{proof}
The complete Gibonacci system mod 4 is

\{(1, 1, 2, 3, 1, 0),

(3, 3, 2, 1, 3, 0),

(2, 2, 0),

(0)\}.

Here each cycle with some $i_j, i_{j+1}$ relatively prime is complete. All Gibonacci sequences must correspond to one of these cycles, so all Gibonacci sequences are complete mod 4. $\blacksquare$
\end{proof}

\begin{lemma}
\label{SixGibComp}
All Gibonacci sequences are complete mod 6.
\end{lemma}

\begin{proof}
The complete Gibonacci system mod 6 is 

\{(1, 1, 2, 3, 5, 2, 1, 3, 4, 1, 5, 0, 5, 5, 4, 3, 1, 4, 5, 3, 2, 5, 1, 0),

(2, 2, 4, 0, 4, 4, 2, 0),

(3, 3, 0),

(0)\}.

Here there is only one cycle with some $i_j, i_{j+1}$ relatively prime, and it is complete, so all Gibonacci sequences are complete mod 6. $\blacksquare$
\end{proof}

\begin{lemma}
\label{FourteenGibComp}
All Gibonacci sequences are complete mod 14.
\end{lemma}

\begin{proof}
The complete Gibonacci system mod 14 contains the Fibonacci cycle mod 14,

\[ \begin{array}{cccccccc}
\mbox{(1,} & 1, & 2, & 3, & 5, & 8, & 13, & 7, \\
\mbox{6,} & 13, & 5, & 4, & 9, & 13, & 8, & 7, \\
\mbox{1,} & 8, & 9, & 3, & 12, & 1, & 13, & 0, \\
\mbox{13,} & 13, & 12, & 11, & 9, & 6, & 1, & 7, \\
\mbox{8,} & 1, & 9, & 10, & 5, & 1, & 6, & 7, \\
\mbox{13,} & 6, & 5, & 11, & 2, & 13, & 1, & 0) \end{array}\]

which has length 48. The system also contains $3 \cdot F_n$ (mod 14) and $5 \cdot F_n$ (mod 14), each of the same length and completeness, totaling $48 \cdot 3 = 144$ terms. Another 52 terms appear in the following cycles:

(2, 2, 4, 6, 10, 2, 12, 0, 12, 12, 10, 8, 4, 12, 2, 0),

(4, 4, 8, 12, 6, 4, 10, 0, 10, 10, 6, 2, 8, 10, 4, 0),

(6, 6, 12, 4, 2, 6, 8, 0, 8, 8, 2, 10, 12, 8, 6),

(7, 7, 0)

(0).

This totals 144 + 52 = 196 terms, a complete Gibonacci system mod 14. Here as well, each cycle with some $i_j, i_{j+1}$ relatively prime is complete, so all Gibonacci sequences are complete mod 14. $\blacksquare$
\end{proof}

\begin{proof}[Proof of Proposition~\ref{LucasMin}]
Once again, because all Gibonacci sequences are complete mod 14, all Gibonacci sequences are also complete mod 2 and 7. This completes our proof of proposition~\ref{LucasMin}. $\blacksquare$
\end{proof}

Thus, to tie everything from sections 2-4 together, we can say by proposition~\ref{FibToGib} that the Fibonacci sequence is complete in the largest set of $m$ possible, by which we mean that the set of numbers in which any generic Gibonacci sequence is complete must be a subset of the set in which the Fibonacci sequence is complete. To this we can add by proposition~\ref{LucasMin} that the Lucas numbers are complete in the smallest possible set of $m$, meaning that the set of numbers in which any generic Gibonacci sequence is complete must contain the set in which the Lucas numbers are complete.

These facts make us consider several things: first, are there other Gibonacci sequences that are complete in the same set of $m$ as the Lucas numbers or the Fibonacci sequence? Then in contrast, are there any Gibonacci sequences which are complete mod $m$ for some set of $m$ between those extremes? We consider these questions in the following section.

\section{Complete-Equivalence of Gibonacci sequences}

In order to consider the set of $m$ in which a generic Gibonacci sequence is complete mod $m$, we would like a way of comparing these sets between different Gibonacci sequences. We accomplish this by defining both a notation for those sets and an equivalence relation between different Gibonacci sequences.

\begin{definition}
For any Gibonacci sequence $\{G_n(a,b)\}$, the set of $m$ in which $\{G_n(a,b)\}$ is complete mod $m$ will be denoted by $M_{(a,b)}$. We also designate $M_F$ = $M_{(1,1)}$ and $M_L$ = $M_{(1,3)}$.
\end{definition}

This notation will allow us to compare two Gibonacci sequences to determine whether their individual sets of complete $m$ are equal.

\begin{definition}
We say that two Gibonacci sequences $\{G_n(a,b)\}$ and $\{H_n(c,d)\}$ are \emph{complete-equivalent} if it is the case that $M_{(a,b)} = M_{(c,d)}$.
\end{definition}

Having stated that complete-equivalence is an equivalence relation, we now prove it below.

\begin{proposition}
\label{EquivRel}
Complete-equivalence is an equivalence relation.
\end{proposition}

\begin{proof}
 Consider the three generic Gibonacci sequences $\{G_n(a,b)\}$, $\{H_n(c,d)\}$, and $\{I_n(e,f)\}$. $\{G_n(a,b)\}$ is complete-equivalent to itself, since $M_{(a,b)} = M_{(a,b)}$. Similarly, if $\{G_n(a,b)\}$ is complete-equivalent to $\{H_n(c,d)\}$, then $M_{(a,b)} = M_{(c,d)}$ and it follows that $M_{(c,d)} = M_{(a,b)}$, meaning $\{H_n(c,d)\}$ is complete-equivalent to $\{G_n(a,b)\}$. Lastly, if $\{G_n(a,b)\}$ is complete-equivalent to $\{H_n(c,d)\}$ and $\{H_n(c,d)\}$ is complete-equivalent to $\{I_n(e,f)\}$, then $M_{(a,b)} = M_{(c,d)}$ and $M_{(c,d)} = M_{(e,f)}$, giving $M_{(a,b)} = M_{(e,f)}$, meaning $\{G_n(a,b)\}$ is complete-equivalent to $\{I_n(e,f)\}$. $\blacksquare$
\end{proof}

For two Gibonacci sequences to satisfy the definition of being complete-equivalent, we must check completeness of each sequence mod $m$ for each $m$. We have already done substantial work in this direction, as proposition~\ref{FibToGib} states that for all $m$ where the Fibonacci sequence is defective mod $m$, any generic Gibonacci sequence will also be defective mod $m$. Thus we only need check those numbers $m$ in which the Fibonacci sequence is complete mod $m$.

Furthermore, proposition~\ref{LucasMin} tells us that if the Lucas numbers are complete mod $m$, then all Gibonacci sequences will also be complete mod $m$. Hence, to show that two Gibonacci sequences satisfy the definition of being complete-equivalent, we need only check those $m$ in which the Fibonacci sequence is complete mod $m$ and the Lucas numbers are defective mod $m$. (See the table below.)

\begin{center}
\begin{tabular}{ | l | p{3cm} | p{3cm} | c | }
\hline                        
\rule{0pt}{3ex}$m$ & {$2, 4, 6, 7, 14, 3^j$ \newline for $j \geq 1$} & {$5^k, 2\cdot 5^k, 4 \cdot 5^k, \newline 3^j \cdot 5^k, 6 \cdot 5^k, \newline 7 \cdot 5^k, 14 \cdot 5^k$ \newline for $j, k \geq 1$} & All others \\ \hline
\rule{0pt}{3ex}$\{G_n(a,b)\}$ complete mod $m$ & All sequences & \{$F_n$\} but not \{$L_n$\} & No sequences  \\
\hline  
\end{tabular}
\end{center}

This means that for all $\{G_n(a,b)\}$, $M_{(a,b)}$ contains the numbers in the second column and none of the numbers in the fourth column. All the numbers of interest to us are therefore those in the center column, and they are all multiples of 5, so the rest of this paper will focus on the completeness or defectiveness of $\{G_n(a,b)\}$ mod 5. We begin by making some comments on the complete Gibonacci system mod 5.

The complete Gibonacci system mod 5 is

\{(1, 1, 2, 3, 0, 3, 3, 1, 4, 0, 4, 4, 3, 2, 0, 2, 2, 4, 1, 0),

(1, 3, 4, 2),

(0)\}.

This complete system contains a stark division that we haven't seen in any other $m$. There are only two Gibonacci cycles mod 5, one corresponding to the Fibonacci sequence, which is complete, and the other corresponding to the Lucas numbers, which is defective. Therefore, any Gibonacci sequence complete-equivalent to the Lucas numbers, if one exists, will necessarily correspond to the same Gibonacci cycle mod 5 as the Lucas numbers. Similarly, a Gibonacci sequence complete-equivalent to the Fibonacci sequence, if one exists, will necessarily correspond to the same Gibonacci cycle mod 5 as the Fibonacci sequence.

However, what we don't yet know is whether the converse holds. Does the defectiveness or completeness of $\{G_n(a,b)\}$ mod 5 determine whether it is complete-equivalent to the Lucas numbers or Fibonacci sequence, respectively? In fact, we can already answer the first half of that question:

\begin{proposition}
\label{FiveDefEquiv}
If $\{G_n(a,b)\}$ has the same Gibonacci cycle (mod 5) as the Lucas numbers, then $\{G_n(a,b)\}$ is complete-equivalent to the Lucas numbers.
\end{proposition}

\begin{proof}
If $\{G_n(a,b)\}$ has the same Gibonacci cycle (mod 5) as the Lucas numbers, then $\{G_n(a,b)\}$ is defective mod 5. By proposition~\ref{mTotm}, $\{G_n(a,b)\}$ is therefore also defective mod $5t$, for all multiples of 5. But as we mentioned above, those are the only numbers we have to check to determine its complete-equivalence. Because $\{G_n(a,b)\}$ is defective mod $5t$ for all $t \in \mathbb{N}$, including those numbers in the third column of the table above, $M_{(a,b)}$ is limited to the numbers in the table's second column. These numbers are all guaranteed to be in $M_{(a,b)}$, and they are the only numbers in $M_L$. Therefore, $M_{(a,b)}$ = $M_L$. $\blacksquare$
\end{proof}

The ease of this proof makes us wonder if we can do the same for the completeness of $\{G_n(a,b)\}$ mod $5$ and its complete-equivalence to the Fibonacci sequence. This would entail showing that if $5 \in M_{(a,b)}$, then all the numbers in the center column of the table above are also in $M_{(a,b)}$. If that is true, then the completeness of $\{G_n(a,b)\}$ mod 5 would give $M_{(a,b)}$ = $M_F$. We will proceed to show exactly this.

To accomplish this we will make use of several results in the existing literature. An important tool in our following proof methods will be what we call the \emph{Gibonacci invariant}, the quantity $|(G_n)^2 + G_n \cdot G_{n+1} - (G_{n+1})^2|$, a previous result in number theory. By calling it an invariant, we mean that the quantity never changes within a Gibonacci sequence no matter what $n$ we choose. We present a proof of this property as our next lemma:

\begin{lemma}
\label{invariant}
The quantity $|(G_n)^2 + G_n \cdot G_{n+1} - (G_{n+1})^2|$ is an invariant for all n in $\{G_n(a,b)\}$.
\end{lemma}

\begin{proof}
We proceed by induction. Let $|(G_n)^2 + G_n \cdot G_{n+1} - (G_{n+1})^2|$ = L. Then we want $|(G_{n+1})^2 + G_{n+1} \cdot G_{n+2} - (G_{n+2})^2| = L.$ But $G_{n+2} = G_n + G_{n+1}$. Making this substitution gives 

$$
\begin{array}{rll}
& &|(G_{n+1})^2 + G_{n+1} \cdot G_{n+2} - (G_{n+2})^2| \\
& = & |(G_{n+1})^2 + G_{n+1} \cdot (G_n + G_{n+1}) - (G_n + G_{n+1})^2| \\
& = & |(G_{n+1})^2 + G_n \cdot G_{n+1} + (G_{n+1})^2 - (G_n)^2 - 2G_n \cdot G_{n+1} - (G_{n+1})^2| \\
& = & |-(G_n)^2 - G_n \cdot G_{n+1} + (G_{n+1})^2| \\
& = & |(G_n)^2 + G_n \cdot G_{n+1} - (G_{n+1})^2| \\
& = & L.
\end{array}
$$

$\blacksquare$
\end{proof}

Because the Gibonacci invariant doesn't change throughout the sequence, we can determine its value for both the Fibonacci sequence and the Lucas numbers. From the Fibonacci sequence we select the consecutive terms 1 and 1, yielding $|1 + 1 - 1| = 1$. From the Lucas numbers we select 1 and 3, yielding $|1 + 4 - 9| = 5$.

If we consider Gibonacci sequences mod 5, we arrive at an interesting result concerning the Gibonacci invariant. Recall from above that the complete Gibonacci system mod 5 is

\{(1, 1, 2, 3, 0, 3, 3, 1, 4, 0, 4, 4, 3, 2, 0, 2, 2, 4, 1, 0),

(1, 3, 4, 2),

(0)\}.

The Gibonacci invariant varies within these cycles because they are representative of some $\{G_n(a,b)\}$ being considered mod 5. For example, consider the Gibonacci cycle corresponding to the Fibonacci sequence. Selecting 4 and 0 gives $|4^2 + 4 \cdot 0 - 0^2|$ = 16, while selecting 1 and 1 gives $|1^2 + 1 \cdot 1 - 1^2|$ = 1. However, its variance within the cycles is more limited if we consider the invariant itself mod $m$, for if $|(G_n)^2 + G_n \cdot G_{n+1} - (G_{n+1})^2|$ = $L$, then $(G_n)^2 + G_n \cdot G_{n+1} - (G_{n+1})^2$ (mod $m$) $\equiv$ $\pm L$ (mod $m$). Hence selecting any two consecutive numbers from the cycle corresponding to the Fibonacci sequence will yield an invariant of $\pm 1$ (mod 5), and selecting any two consecutive numbers from the cycle corresponding to the Lucas numbers will yield an invariant of 0 (mod 5). Because these are the only two Gibonacci cycles mod 5 other than the trivial cycle, all Gibonacci cycles must have a Gibonacci invariant of either 0 or $\pm 1$ when the invariant is considered mod 5. This is convenient, because it gives us an easy way to check whether a Gibonacci sequence is complete or defective mod 5. If the Gibonacci invariant is 0 (mod 5), then the sequence has the same Gibonacci cycle mod 5 as the Lucas numbers and is hence defective mod 5. Conversely, if the Gibonacci invariant is $\pm 1$ (mod 5), then the sequence has the same Gibonacci cycle mod 5 as the Fibonacci sequence and is hence complete mod 5. We present our final theorem based on these observations.

\begin{theorem}
\label{InvarComp}
If $a^2 + ab - b^2 \equiv 0$ (mod 5), then $M_{(a,b)} = M_L$. Otherwise, $M_{(a,b)} = M_F$.
\end{theorem}

We already have the first half of this theorem from lemma~\ref{FiveDefEquiv}, because if $\{G_n(a,b)\}$ has a Gibonacci invariant congruent to 0 (mod 5), then it has the same Gibonacci cycle mod 5 as the Lucas numbers, which is what lemma~\ref{FiveDefEquiv} assumes. To show the second half, we require two results of Wall~\cite{Wall60} and one of Burr's lemmas, so we reproduce them in our own words below:

\begin{proof}[Theorem 2 (Wall)]
If $m$ has the prime factorization $m = \prod{{p_i}^{e_i}}$ and if $h_i$ denotes the length of the cycle of $\{G_n(a,b)\}$ mod ${p_i}^{e_i}$, then the length of the cycle of $\{G_n(a,b)\}$ mod $m$ is the least common multiple of the $h_i$.
\end{proof}

This theorem says that if we know the lengths of a Gibonacci sequence for each prime power in the prime factorization of $m$, then we can determine the length of that Gibonacci sequence mod $m$, as well. Specifically of interest to us is the length of a Gibonacci sequence mod $5^k$, since all $m$ which we are still interested in are multiples of 5. This leads us to the next result of Wall to be considered.

\begin{proof}[Theorem 9 (Wall)]
If $m = 5^k$, then $\{G_n(a,b)\}$ either has the same length mod $m$ as the Fibonacci sequence, or it has a length which is a fifth of the length of the Fibonacci sequence mod $m$, determined by whether or not the Gibonacci invariant is divisible by 5.
\end{proof}

This means that for any Gibonacci sequence $\{G_n(a,b)\}$ with $a^2 + ab - b^2 \equiv 0$ (mod 5), the length of $\{G_n(a,b)\}$  mod $5^k$ is shorter than the Fibonacci cycle mod $5^k$ by a factor of 5. Otherwise, $\{G_n(a,b)\}$ mod $5^k$ will have the same length as the Fibonacci cycle mod $5^k$. We have already seen this in the mod 5 case. The Lucas numbers have a Gibonacci cycle of length 4, which is one fifth of 20, the length of the Fibonacci cycle mod 5.

We should also note that this applies to all $t \cdot 5^k$ for $t \in \{2, 4, 6, 7, 14, 3^j\}$. We can make that assertion because the length of the Gibonacci cycles is the same within each of those mods regardless of $a^2 + ab - b^2$ (mod 5). We observe from our previous construction of complete Gibonacci systems that $\{G_n(a,b)\}$ has length 3 mod 2, length 6 mod 4, length 24 mod 6, and length 48 mod 14. $\{G_n(a,b)\}$ also has length 16 mod 7 (following from Wall's Theorem 2), and length $8*3^{j-1}$ mod $3^j$ (a result of Burr).

The last piece we must cover is Burr's extension of the completeness of the Fibonacci sequence mod 5 to its completeness mod $5t$. Burr does this by proving his own lemma, which we restate in our own words here:

\begin{proof}[Lemma 3 (Burr)]
Suppose that the Fibonacci cycle mod $m$ has length $k$, and that it has length $5k$ mod $5m$. For some $n$ and $a$ let $F_n \equiv a$ (mod $m$). Then $F_n, F_{k+n}, ..., F_{4k+n}$ are congruent to $a, m+a, ..., 4m+a$ (mod $5m$) in some order.
\end{proof}

We present Burr's proof of this lemma in Appendix A. The lemma, along with Wall's results, says that if the Fibonacci cycle mod 5 contains $a$, then the Fibonacci cycle mod $5m$ will contain $a, m+a, ..., 4m+a$. Hence, to put it another way, because the Fibonacci sequence is complete mod 5, if the conditions on $m$ and $k$ in Burr's lemma hold, then the Fibonacci sequence will be complete mod $5m$. This is exactly what we want to say about any Gibonacci sequence which has the same Gibonacci cycle mod $5$ as the Fibonacci sequence. We now proceed with these final proofs, working in cases depending on whether or not $5 \mid m$:

\begin{lemma}
\label{FiveToFiveM}
Let $a^2 + ab - b^2 \not\equiv 0$ (mod 5), and assume $\{G_n(a,b)\}$ is complete mod $m$ with $m \not\equiv 0$ (mod 5). Then $\{G_n(a,b)\}$ is also complete mod $5m$.
\end{lemma}

\begin{proof}
Because $\{G_n(a,b)\}$ has a Gibonacci invariant which is not congruent to 0 (mod 5), by Wall's Theorem 9 we know that $\{G_n(a,b)\}$ mod 5 has the same length as the Fibonacci sequence mod 5. Burr, in showing this result for the Fibonacci sequence, solely relies on the nature of the Gibonacci cycle mod 5 corresponding to the Fibonacci sequence mod 5 (see Appendix A). This property holds for $\{G_n(a,b)\}$ as well, since $\{G_n(a,b)\}$ has the same Gibonacci cycle mod 5 as the Fibonacci sequence. Therefore, Burr's proof applies directly to our lemma. $\blacksquare$
\end{proof}

\begin{lemma}
\label{AnotherFiveToFiveM}
Let $a^2 + ab - b^2 \not\equiv 0$ (mod 5), and assume $\{G_n(a,b)\}$ is complete mod $m$ with $m \equiv 0$ (mod 5). Then $\{G_n(a,b)\}$ is complete mod $5m$ as well.
\end{lemma}

\begin{proof}
As in the previous lemma, $\{G_n(a,b)\}$ mod $5^k$ has the same length as the Fibonacci sequence mod $5^k$ because $a^2 + ab - b^2 \not\equiv 0$ (mod 5). Burr's proof for the Fibonacci sequence relies on the lengths of Gibonacci cycles corresponding to the Fibonacci sequence and the fact that the Gibonacci invariant is not divisible by 5 in the Fibonacci sequence (see Appendix A). Since $\{G_n(a,b)\}$ has the same Gibonacci cycle mod 5 as the Fibonacci sequence, Wall's Theorems 2 and 9 tell us that the length of $\{G_n(a,b)\}$ mod $5m$ will be the same as the length of the Fibonacci sequence mod $5m$. Hence, we only need the observation that any Gibonacci sequence having the same Gibonacci cycle mod 5 as the Fibonacci sequence mod 5 has a Gibonacci invariant not divisible by 5, which we have by hypothesis, and Burr's proof applies directly to this lemma as well. $\blacksquare$
\end{proof}

\begin{proof}[Proof of Theorem~\ref{InvarComp}]
If $a^2 + ab - b^2 \equiv 0$ (mod 5), then by lemma~\ref{FiveDefEquiv} we have $M_{(a,b)} = M_L$. Otherwise, by lemmas~\ref{FiveToFiveM} and ~\ref{AnotherFiveToFiveM}, we have $M_{(a,b)} = M_L$. $\blacksquare$
\end{proof}

Hence, to determine $M_{(a,b)}$, we need only consider whether $a^2 + ab - b^2 \equiv 0$ (mod 5). We have thus completed our categorization of which $m$ have the property that $\{G_n(a,b)\}$ is complete mod $m$.

\section{Further Generalizations}

Having categorized all $M_{(a,b)}$, we can move on to other generalizations of Fibonacci-like sequences. We will confine our discussion herein to what we have termed as \emph{Tribonacci sequences}, sequences where each term is the sum of the last \emph{three} terms. These sequences are similar to Gibonacci sequences in some ways and different from them in others. This section will explore some of those similarities and differences. The first similarity is their definition, which is much the same as Gibonacci sequences:

\begin{definition}
A \emph{Tribonacci sequence} is the numerical sequence $\{T_n(a, b, c)\}$ where $a, b, c \in \mathbb{Z}$, satisfying the three following conditions:
\begin{enumerate}
  \item $\gcd(a, b, c) = 1$.
  \item $T_1(a, b, c) = a, T_2(a, b, c) = b$, and $T_3(a, b, c) = c$.
  \item $T_{n+1}(a, b, c) = T_{n-2}(a, b, c) + T_{n-1}(a, b, c) + T_n(a, b, c)$.
\end{enumerate}
\end{definition}

For example, $T_n(1, 1, 1) = \{1, 1, 1, 3, 5, 9, 17, ...\}$ and $T_n(1, 1, 2) = \{1, 1, 2, 4, 7, 13, ...\}$. We will refer to these sequences throughout this section, so we will let $T_n(1, 1, 1) = A_n$ and $T_n(1, 1, 2) = B_n$.

Individual Tribonacci sequences behave much like individual Gibonacci sequences. They are cyclical, with a proof similar to the proof for proposition~\ref{GibCyclical}, using repetition of some triplet mod $m$ instead of some pair. This allows for Tribonacci cycles:

\begin{definition}
We say the finite sequence of integers $(i_1, i_2, ..., i_{h-1}, i_h)$ (mod $m$) is a \emph{Tribonacci cycle mod m} if it satisfies the following conditions:

\begin{enumerate}
  \item $i_{n+3} \equiv i_n + i_{n+1} + i_{n+2}$ (mod $m$) for $n = 1, 2, ..., h-3$.
  \item $i_{h-2} + i_{h-1} + i_h \equiv i_1$ (mod $m$), $i_{h-1}+ i_h + i_1 \equiv i_2$ (mod $m$), and $i_h+ i_1 + i_2 \equiv i_3$ (mod $m$).
  \item No $n < h$ satisfies the relationships in condition 2.
\end{enumerate}
\end{definition}

We will begin our exploration of the differences between Tribonacci sequences and Gibonacci sequences with a consideration of Tribonacci sequences mod 2. $A_n$ (mod 2) is (1), since no even number ever appears, and $B_n$ (mod 2) is (1, 1, 0, 0). This is quite different from our discussion of Gibonacci sequences, where all Gibonacci sequences were complete mod 2. Here $A_n$ is defective mod 2, while $B_n$ is complete.

We can also define what it means for Tribonacci cycles to be equivalent:

\begin{definition}
We say two Tribonacci cycles $(i_1, i_2, ..., i_{h-1}, i_h)$ (mod $m$) and $(j_1, j_2, ..., j_{h'-1}, j_{h'})$ (mod $m$) are \emph{equivalent} if

\begin{enumerate}
\item $h = h'$.
\item For some $0 \leq n, r \leq h$, $i_n = j_r$, $i_{n+1 (mod \ h)} = j_{r+1 (mod \ h)}$, and $i_{n+2 (mod \ h)} = j_{r+2 (mod \ h)}$.
\end{enumerate}
\end{definition}

So (1, 1, 1, 0, 2, 0, 2, 1, 0, 0) = (0, 2, 1, 0, 0, 1, 1, 1, 0, 2) since they are the same lengths and contain the same numbers, just shifted over in the cycle.

We can also create a set of all inequivalent Tribonacci cycles mod $m$.

\begin{definition}
A \emph{complete Tribonacci system mod $m$} is the set of all inequivalent Tribonacci cycles mod $m$.
\end{definition}

A complete Tribonacci system mod $m$ contains a total of $m^3$ terms. This can be proved with a proof similar to the proof of proposition~\ref{MSquared}, taking all possible triplets $[a, b, c]$ and forming cycles by merging $[a, b, c]$ to $[b, c, d]$ if $a + b + c \equiv d$ (mod $m$). As this process continues, each triplet would effectively merge away two of its terms, leaving $3m^3 - 2m^3 = m^3$ terms. For example, the complete Tribonacci system mod 2 is \\ \{(0), (1), (1, 1, 0, 0), (1, 0)\}, with 8 terms.

One of our first considerations for Tribonacci sequences is whether or not there are corollaries of the Fibonacci sequence and the Lucas numbers, i.e., two Tribonacci sequences whose sets of complete $m$ form a subset and superset of all complete $m$. However, this does not appear to be the case. Consider $B_n$ mod 9: It has the Tribonacci cycle

\[\begin{array}{ccccccccccccc}
\mbox{(1,} & 1, & 2, & 4, & 7, & 4, & 6, & 8, & 0, & 5, & 4, & 0, & 0, \\
\mbox{4,} & 4, & 8, & 7, & 1, & 7, & 6, & 5, & 0, & 2, & 7, & 0, & 0, \\
\mbox{7,} & 7, & 5, & 1, & 4, & 1, & 6, & 2, & 0, & 8, & 1, & 0, & 0).\end{array}\]

It is defective because it doesn't contain 3. However, $A_n$ mod 9 =

\[\begin{array}{ccccccccccccc}
\mbox{(1,} & 1, & 1, & 3, & 5, & 0, & 8, & 4, & 3, & 6, & 4, & 4, & 5, \\
\mbox{4,} & 4, & 4, & 3, & 2, & 0, & 5, & 7, & 3, & 6, & 7, & 7, & 2, \\
\mbox{7,} & 7, & 7, & 3, & 8, & 0, & 2, & 1, & 3, & 6, & 1, & 1, & 8).\end{array}\]

This contains all residues, so $A_n$ is complete mod 9. Thus $B_n$ is complete mod 2 and defective mod 9, while $A_n$ is defective mod 2 and complete mod 9. This reversal of completeness and defectiveness never occurred with any Gibonacci sequence. It doesn't make the existence of two Tribonacci sequences similar to the Fibonacci sequence and the Lucas numbers impossible, since some sequence could be both complete mod 2 and 9 or defective mod 2 and 9, but it does make their existence unlikely. For example, $T_n(1, 2, 3)$ is both complete mod 2 and mod 9, but is defective mod 67, while both $A_n$ and $B_n$ are complete mod 67. This gives us enough confidence to make the following conjecture, using $M_{(a, b, c)}$ as the set of $m$ in which $\{T_n(a,b,c)\}$ is complete mod $m$:

\begin{proof}[Conjecture]
There is no $a, b, c$ so that $M_{(a,b,c)} \subset M_{(a', b', c')}$ or $M_{(a', b', c')} \subset M_{(a,b,c)}$ for all $a', b', c'$.
\end{proof}

An important characteristic of complete Gibonacci systems was the fact that no $\{G_n(a,b)\}$ was complete mod $p$ for any prime $p$ larger than 7. Thus the discussion became a discussion of composite numbers. This is perhaps the most noteworthy difference between Tribonacci sequences and Gibonacci sequences. Using a computer program, it has been determined that both $A_n$ and $B_n$ are complete mod $p$ in approximately 61\% of the first 300 primes (see tables below). There appears to be some correlation between $A_n$ and $B_n$ as to the $p$ in which they are complete mod $p$, but as often as not one will be complete and the other will be defective. The percentage is also not strictly decreasing, as $A_n$ is complete in 59\% of the primes from 547 to 1151 (the 101st to 200th), but is complete in 63\% of the primes from 1153 to 1987 (the 201st to 300th).

\begin{center}
\begin{tabular}{ | l | p{4cm} | c | p{4cm} | c | }
\hline                        
Primes $p$ by 20s& Number of $p$ such that $A_n$ is complete mod $p$& \% & Number of $p$ such that $B_n$ is complete mod $p$& \% \\ \hline
2 - 71& 12 & 60\% & 16 & 80\% \\ \hline
73 - 173& 11 & 55\% & 12 & 60\% \\ \hline
179 - 281& 11 & 55\% & 12 & 60\% \\ \hline
283 - 409& 14 & 70\% & 12 & 60\% \\ \hline
419 - 541& 14 & 70\% & 12 & 60\% \\ \hline
547 - 659& 12 & 60\% & 12 & 60\% \\ \hline
661 - 809& 12 & 60\% & 11 & 55\% \\ \hline
811 - 941& 11 & 55\% & 12 & 60\% \\ \hline
947 - 1069& 16 & 80\% & 14 & 70\% \\ \hline
1087 - 1223& 8 & 40\% & 10 & 50\% \\ \hline
1229 - 1373& 10 & 50\% & 11 & 55\% \\ \hline
1381 - 1511& 14 & 70\% & 12 & 60\% \\ \hline
1523 - 1657& 9 & 45\% & 12 & 60\% \\ \hline
1663 - 1811& 15 & 75\% & 12 & 60\% \\ \hline
1823 - 1987& 15 & 75\% & 14 & 70\% \\ \hline 
\end{tabular}
\end{center}

\begin{center}
\begin{tabular}{ | l | p{4cm} | c | p{4cm} | c | }
\hline                        
By 100s & Number of $p$ such that $A_n$ is complete mod $p$& \% & Number of $p$ such that $B_n$ is complete mod $p$& \% \\ \hline
2 - 541& 62 & 62\% & 64 & 64\% \\ \hline
547 - 1151& 59 & 59\% & 59 & 59\% \\ \hline
1153 - 1987& 63 & 63\% & 61 & 61\% \\ \hline
\end{tabular}
\end{center}

\begin{center}
\begin{tabular}{ | l | p{4cm} | c | p{4cm} | c | }
\hline                        
& Number of $p$ such that $A_n$ is complete mod $p$& \% & Number of $p$ such that $B_n$ is complete mod $p$& \% \\ \hline
2 - 541& 62 & 62\% & 64 & 64\% \\ \hline
2 - 1151& 121 & 60.5\% & 123 & 61.5\% \\ \hline
2 - 1987&  184 & 61.3\% & 184 & 61.3\% \\ \hline
\end{tabular}
\end{center}

This is not an adequate sample to make the conjecture that the percentage approaches 61\%, but we see enough defectiveness to offer the following conjecture:

\begin{proof}[Conjecture]
Given $\{T_n(a,b,c)\}$, let $P_n$ be the set of the first $n$ primes and let $C_n = M_{(a,b,c)} \cap P_n$. Then $\lim_{n\to\infty}\frac{|C_n|}{|P_n|} < 1$.
\end{proof}

A point to consider that may be of assistance in future work is the lengths of Tribonacci cycles. In the Fibonacci sequence, $\{F_n\}$ had length $p-1$ for all primes $p \equiv \pm 1$ (mod 10) (a result of Wall). This made all such primes defective, since their cycle length was less than the number of residues required for completeness. This phenomenon of primes $p$ having length $p-1$ occurs similarly in Tribonacci sequences, but not with the same regularity or frequency. If some pattern could be determined for which primes are defective in this way, a conjecture could be made about an upper limit for the percentage of primes that are complete for a given Tribonacci sequence.

If future work in the area of Tribonacci sequences is to follow the pattern of Gibonacci sequences, then the following points should be considered:

\begin{enumerate}
\item{Work should be started with one individual Tribonacci sequence. The wealth of knowledge surrounding the Fibonacci sequence made it a good place to start a discussion of Gibonacci sequences, and that knowledge carried into this paper. We suggest the sequence $B_n = \{1, 1, 2, 4, 7, 11, 18, ...\}$. Extrapolation backwards shows that it contains two zeros, since 2 - 1- 1 = 0 and 1 - 1 - 0 = 0. The Fibonacci sequence was the only Gibonacci sequence that contained 1 zero, since any other sequence with zero in it would be a multiple of the Fibonacci sequence and hence not a Gibonacci sequence. The same will hold true for $B_n$, which will be the only Tribonacci sequence containing two consecutive zeros. Thus it is likely the Tribonacci sequence most similar to the Fibonacci sequence.}
\item{Special attention should be given to the length of the sequence's cycles mod $p$ for prime $p$. When this is determined, the next step is to consider powers of primes and whether the sequence is complete mod $p^j$ for all $j$. For example, both $A_n$ and $B_n$ are complete mod $5, 5^2, 5^3, 5^4$ and $5^5$. This could lead one to conjecture that they are complete for all $5^j$, and a study of complete Tribonacci systems mod $5^j$ would aid in that endeavor.}
\item{Work should be done to determine whether or not there is a Tribonacci invariant, and if so, what form it takes. This would be much like the Gibonacci invariant, a quantity based on $n$ that holds true for all $T_n$ in each individual Tribonacci sequence. This could prove useful in the work for both of the previously mentioned points.}
\item{Once a large amount is known about the completeness or defectiveness of one sequence mod $m$, that information can be used to aid in determining the completeness or defectiveness of other sequences mod $m$, as this paper has done with Gibonacci sequences.}
\end{enumerate}

\pagebreak

\section{Appendix A: Burr's Lemma 3}

Here we present Burr's proof of Lemma 3 with our own additions for the sake of clarity. We have modified Burr's text to correct several misprints and to match the language we have used in this paper. Our comments will appear as indented, italicized blocks.

\begin{proof}[Lemma 3 (Burr)]
Suppose that the Fibonacci cycle mod $m$ has length $k$, and that it has length $5k$ mod $5m$. For some $n$ and $a$ let $F_n \equiv a$ (mod $m$). Then $F_n, F_{k+n}, ..., F_{4k+n}$ are congruent to $a, m+a, ..., 4m+a$ (mod $5m$) in some order.
\end{proof}

\begin{proof}
We consider two cases, depending on whether or not $5 \mid m$. We first assume $5 \nmid m$. Then the length of $F_n$ mod $5m$ is the lcm of $k$ and the period of 5, which is 20.

\begin{quote}
\emph{This follows directly from Wall's Theorem 2.}
\end{quote}

Since this length is to equal $5k$, we have $k \equiv 4, 8, 12, 16$ (mod 20).

\begin{quote}
\emph{We have lcm(20, $k$)= $5k$ by hypothesis, which means that $k$ must have a factor of 4 mod 20. We reject $k \equiv 0$ (mod 20) because then lcm(20, $k$) = $k$, not $5k$.}
\end{quote}

Now, the Fibonacci cycle mod 5 is (0, 1, 1, 2, 3, 0, 3, 3, 1, 4, 0, 4, 4, 3, 2, 0, 2, 2, 4, 1). From this it may be verified that $F_n, F_{k+n}, ..., F_{4k+n}$ are congruent modulo 5 to 0, 1, 2, 3, 4 in some order. For instance, if $n \equiv 0$ (mod 20) they are congruent respectively to 0, 3, 1, 4, 2. Since each of these is congruent to $a$ modulo $m$, they are congruent in some order to $a, m+a, ..., 4m+a$. This completes the first case.

\begin{quote}
\emph{The assertion in the first half of this paragraph follows from arithmetical observation of the Fibonacci cycle mod 5. Pick some $k \equiv 4, 8, 12$, or 16 (mod 20) and some $n$, and the set $F_n, F_{k+n}, ..., F_{4k+n}$ will be congruent to 0, 1, 2, 3, 4 (mod 5) in some order.}

\emph{The concluding assertion follows because each of $F_n, F_{k+n}, ..., F_{4k+n}$ is congruent to $a$ mod $m$. This is given, since by hypothesis $\{F_n\}$ has length $k$ mod $m$, meaning if we add $k$ to any indice $n$, then $F_n \equiv F_{k+n}$ (mod $m$). Combining this with the previous statement, we get $F_n, F_{k+n}, ..., F_{4k+n}$ congruent to $a, m+a, ..., 4m+a$ (mod $5m$) in some order.}

\emph{Thus the proof of this case solely relies on the nature of the Fibonacci cycle mod 5, and we can apply it to any Gibonacci sequence which has the same Gibonacci cycle mod 5 as the Fibonacci sequence.}
\end{quote}

We now assume $5 \mid m$. Since the Fibonacci sequence has length $k$ modulo $m$, $F_n, F_{k+n}, ..., F_{4k+n}$ are all congruent to $a$ mod $m$ and hence are each congruent to $im + a$ mod $5m$ for some choice of $0 \leq i \leq 4$.

\begin{quote}
\emph{Each must take the form $im + a$, but we don't have uniqueness of $i$ since 5 and $m$ are not relatively prime.}
\end{quote}

Our object is to show that the value of $i$ is different for each of the five terms. Set $F_{n+1} \equiv b$ (mod $m$). Then $F_{n+1}, F_{k+n+1}, ..., F_{4k+n+1}$ are each congruent to $jm + b$ mod $5m$ for some $0 \leq j \leq 4$. Speaking in terms of the concept we have defined, there are 25 pairs congruent modulo $5m$ to $(im + a, jm + b)$ appearing within a complete Gibonacci cycle mod $5m$, of which 5 appear in the cycle corresponding to the standard Fibonacci sequence.

\begin{quote}
\emph{The conditions on $j$ follow similarly as for $i$. There are 25 pairs congruent modulo $5m$ to $(im + a, jm + b)$ because there are 5 possible choices for both $i$ and $j$, and $5 \cdot 5 = 25$. The fact that 5 of them appear in the standard Fibonacci sequence mod $5m$ follows because we have the pair ($a, b$) in the Fibonacci cycle mod $m$. By hypothesis, the length of the Fibonacci cycle mod $5m$ is 5 times as long as its length mod $m$, so this pair repeats 5 times within the Fibonacci cycle mod $5m$ when that cycle is considered mod $m$. Thus there are 5 pairs taking the form $(im + a, jm + b)$ mod $5m$ within the Fibonacci cycle mod $5m$.}

\emph{This will also follow for any Gibonacci sequence of the same length as the Fibonacci sequence mod $5m$.}
\end{quote}

Our objective is to show that each of these 5 gives a different value of $i$.

Since

\begin{equation}
a^2 + ab - b^2 \equiv \pm 1 (mod \ m),
\end{equation}

\begin{quote}
\emph{This follows since the Fibonacci invariant is 1.}
\end{quote}

We may set

\begin{equation}
a^2 + ab - b^2 = mA \pm 1.
\end{equation}

\begin{quote}
\emph{For some $A$.}
\end{quote}

Applying this same invariant to the pair ($im + a, jm + b$), we have

$$
\begin{array}{rll}
& &(im + a)^2 + (im + a)(jm + b) - (jm + b)^2 \\
& = & i^2m^2 + ijm^2 -j^2m^2 + ((2a + b)i + (a - 2b)j)m + a^2 +ab - b^2 \\
& = & m^2(i^2 + ij - j^2) + m((2a + b)i + (a - 2b)j) + mA \pm 1. \\
\end{array}
$$

\begin{quote}
\emph{Through multiplication, simplification and substitution of $mA \pm 1$.}
\end{quote}

The last expression will be congruent to $\pm 1$ (mod $5m$) if and only if

\begin{equation}
(2a + b)i + (a - 2b)j + A \equiv 0 \ (mod \ 5).
\end{equation}

\begin{quote}
\emph{We know $m^2(i^2 + ij - j^2)$ is congruent to 0 mod $5m$ since $5 \mid m$, giving $5m \mid m^2$. Thus the whole expression will be congruent to $\pm 1$ mod $5m$ if and only if $m((2a + b)i + (a - 2b)j) + mA \equiv 0$ (mod $5m$), which is equivalent to saying $(2a + b)i + (a - 2b)j + A \equiv 0$ (mod 5).}

\emph{This argument will also apply so long as $a^2 + ab - b^2 \not\equiv 0$ (mod 5), since the $\pm 1$ would just be replaced by $\pm L$ (mod m), where $L$ is the Gibonacci invariant of $\{G_n(a,b)\}$}
\end{quote}

However, $2a + b \not\equiv 0$ (mod 5) since otherwise

\begin{equation}
\pm 1 \equiv a^2 + ab - b^2 \equiv a^2 - 2a^2 -4a^2 \equiv 0 \ (mod \ 5);
\end{equation}

similarly $a - 2b \not\equiv 0$ (mod 5).

\begin{quote}
\emph{This derives a contradiction through direct substitution of $-2a$ for $b$ into the Fibonacci invariant in the former case, and $2b$ for $a$ in the latter case.}
\end{quote}

Consequently, for each of the 5 possible choices of $i$, there is exactly one $j$ satisfying the above congruence. Hence only these 5 pairs could appear as consecutive pairs in the Fibonacci sequence. Since $i$ is different in each case, the lemma is proved.

\begin{quote}
\emph{Because $(2a + b)$ and $(a - 2b)$ are not congruent to 0 mod 5, there is no way to cancel the $i$ and $j$ out of the expression in question. Thus, for each $i$, we have $j$ as unique for a previously determined $a, b,$ and $A$. If the same $i$ appears more than once, then its second occurrence will have the same $j$ as the previous occurrence, and hence the Fibonacci cycle mod $5m$ will repeat before its length of $5k$. Thus $i$ is different for each of the 5 pairs.} $\blacksquare$
\end{quote}
\end{proof}

\pagebreak

\end{document}